\documentclass[10pt]{amsart}
\usepackage{amsthm}
\usepackage{amsmath}
\usepackage{amsfonts}
\usepackage{amssymb}

\newcommand\mysup[1]{^{(#1)}}

\def\js{\mathcal{C}\mysup{k}_{2,n}}
\def\jstwo{\mathcal{C}\mysup{k}_{2,2}}
\def\jsthree{\mathcal{C}\mysup{k}_{2,3}}

\def\cjs{\mathcal{V}(\mathcal{C}\mysup{k}_{2,n})}
\def\cjsm{\mathcal{V}(\mathcal{C}\mysup{k}_{2,m})}
\def\cjsu{\mathcal{V}(\mathcal{C}\mysup{k}_{2,u})}
\def\cjsv{\mathcal{V}(\mathcal{C}\mysup{k}_{2,v})}
\def\cjones{\mathcal{V}(\mathcal{C}\mysup{1}_{2,n})}
\def\cjsthree{\mathcal{V}(\mathcal{C}\mysup{k}_{2,3})}

\def\disc{\Delta_k(a)}
\def\Ubar{\overline{U}}
\def\Spec{\text{Spec}}
\def\U{\mathcal{U}}

\begin{document}

\renewcommand{\baselinestretch}{1.1}

\newtheorem{Theorem}{Theorem}[section]
\newtheorem{Lemma}[Theorem]{Lemma}
\newtheorem{Proposition}[Theorem]{Proposition}
\newtheorem{Corollary}[Theorem]{Corollary}

\theoremstyle{remark}
\newtheorem{Remark}[Theorem]{Remark} 
\newtheorem{Notation}[Theorem]{Notation}
\newtheorem{Definition}[Theorem]{Definition}
\newtheorem{Note}[Theorem]{Note}
\newtheorem{Example}[Theorem]{Example}

\author{B.A. Sethuraman}
\author{Klemen \v{S}ivic}
\thanks{First author supported by  National Science Foundation grant DMS-0700904.  Second author supported by Slovenian Research Agency.}
\address{Dept. of Mathematics\\California State University
Northridge\\Northridge CA 91330\\U.S.A.}
\address{Institute of Mathematics, Physics, and Mechanics\\University of Ljubljana\\Jadranska
19\\ 1000 Ljubljana\\Slovenia}
\title{Jet Schemes of the Commuting Matrix Pairs Scheme}
\email{al.sethuraman@csun.edu} \email{klemen.sivic@fmf.uni-lj.si}

\begin{abstract}
We show that for all $k\ge 1$, there exists an integer $N(k)$ such
that for all $n\ge N(k)$ the $k$-th order jet scheme over the
commuting $n\times n$ matrix pairs  scheme is reducible.
At the other end of the spectrum, it is known that for all $k\ge
1$, the $k$-th order jet scheme over the commuting $2\times 2$
matrices is irreducible: we show that the same holds for $n=3$.
\end{abstract}

\maketitle


\section  {Introduction}  \label{secn:intro}
Recall that if $F$ is an algebraically closed field, 
$\{
x_{i,j},\ y_{i,j}, 1\le i,j\le n\}$ are variables, and $X =
(x_{i,j})$, $Y=(y_{i,j})$, the commuting $n\times n$ matrix pairs
scheme $\mathcal{C}_{2,n}$ is $\Spec(F[\{ x_{i,j},\ y_{i,j}\}]/J)$,
where $J$ is the ideal generated by the $n^2$ components of the
matrix $XY-YX$.  It is known that $\mathcal{C}_{2,n}$ is
irreducible of dimension $n^2+n$ (see \cite{MOTT} or
\cite{Gerst}), but it is an open problem if it is reduced.

For any scheme $X$ of finite type over $F$, we may define the
scheme of $k$-th order jets $X\mysup{k}$, whose closed points are
in bijection with all morphisms $\Spec(F[t]/t^{k+1})
\rightarrow X$ (see e.g. \cite[\S 2]{M2}). When $X$ is affine,
$X\mysup{k}$ is also affine.  In particular, for the affine scheme
 $\mathcal{C}_{2,n}$, $\js$ may be identified with
$\Spec(R)$ where $R$ is defined as follows: Let $t$,
$x\mysup{s}_{i,j}$, and $y\mysup{s}_{i,j}$ ($0\le s \le k$, $1\le i,j
\le n$)  be new variables. For each generator $f_u$ of
$J$, replace each $x_{i,j}$ and $y_{i,j}$ by $\sum_{s=0}^k
x\mysup{s}_{i,j} t^s$ and $\sum_{s=0}^k y\mysup{s}_{i,j} t^s$
respectively, and expand $f_u$ as $\sum_{s=0}^k f\mysup{s}_{u}
t^s$ mod $t^{k+1}$. Let $J\mysup{k}$  be the ideal of $F[\{x\mysup{s}_{i,j},y\mysup{s}_{i,j}\}]$ generated by
$f\mysup{0}_u$, $\dots$, $f\mysup{k}_u$, $u = 1,\dots, n^2$. Then
$R = F[\{x\mysup{s}_{i,j},y\mysup{s}_{i,j}\}]/J\mysup{k}$. (Intuitively,
the closed points of  $\js$
yield paremeterized curves in $\mathbf{A}^{2n^2}$ of degree $k$
that vanish to degree $k$ at closed points of $X$.)

Interest in jet schemes arises from the connections between the
singularities of $X$ and the irreducibility of $X\mysup{k}$.  For
instance, Musta\c{t}\v{a} (\cite{M1}) shows that for $X$ a locally
complete intersection, $X$ has rational singularities if and only
if all jet schemes $X\mysup{k}$ are irreducible. Much earlier,
Nash (\cite{N1}) related the arc space of $X$ (the projective
limit of the various $X\mysup{k}$) to the exceptional divisors
over the singular points of $X$ in a resolution of singularities.

Our   interest in the specific jet schemes $\js$ arises from an
open problem in commuting matrices. It is unknown whether the
algebra
 $F[A,B,C]$ generated by three commuting $n\times n$ matrices $A$,
 $B$, and $C$,
has dimension bounded by $n$. (The corresponding answer is yes for
algebras generated by two commuting matrices, and no for the
algebra generated by four or more commuting matrices, see
\cite{Gur} for instance.) It is natural, while attacking this
problem, to fix $C$ to be of various special forms, and to study
the set of commuting pairs of matrices $(A,B)$ that lie in the
centralizer of $C$.  Write $J_{k+1}$ for the nilpotent Jordan
block of size $k+1$. When $C$ has the special form consisting of
$n$ copies of  $J_{k+1}$ along the diagonal, the centralizer of
$C$ in $M_{n(k+1)}(F)$ is isomorphic to $M_n(F[t])/t^{k+1} =
M_n(F) [t]/t^{k+1}$ (with $C$ corresponding to $t$), and the
commuting pairs $(A,B)$ in the centralizer of $C$ are just pairs
$A(t)=A_0 + A_1 t + \cdots + A_k t^k$, and $B(t)=B_0 + B_1 t +
\cdots + B_k t^k$ that commute in $M_n(F)[t]/t^{k+1}$.

(This is elementary, and can be gleaned from the Toeplitz
structure of the matrices that commute with $C$--see e.g.,
\cite{Gant}. Any matrix $A \in M_{n(k+1)}(F)$ that commutes with
$C$ is of the block form $(A_{i,j})_{i,j}$, where each $A_{i,j}$
is a $(k+1)\times (k+1)$ matrix of the form
\begin{equation*}
\left(%
\begin{array}{ccccc}
  x\mysup{0}_{i,j} & x\mysup{1}_{i,j} & x\mysup{2}_{i,j} & \dots & x\mysup{k}_{i,j} \\
  0 & x\mysup{0}_{i,j} & x\mysup{1}_{i,j}& x\mysup{2}_{i,j} & \vdots \\
  \vdots & \ddots & \ddots & \ddots & \ddots \\
  \vdots & \ddots & \ddots & \ddots & x\mysup{1}_{i,j} \\
  0 & \dots &\dots & 0 & x\mysup{0}_{i,j} \\
\end{array}%
\right)
\end{equation*}

 The $n\times n$ matrix $A_0$ is defined as $(x\mysup{0}_{i,j})_{i,j}$,
 $A_1$ as  $(x\mysup{1}_{i,j})_{i,j}$, etc. We will routinely identify
the centralizer of $C$ in $M_{n(k+1)}(F)$  with $M_n(F)
[t]/t^{k+1}$.)

It follows by expanding the commutator $[A(t),B(t)]$ in powers of $t$
that the equations for $A(t)$ and $B(t)$ to commute are just the
generators of $J\mysup{k}$ described above. Hence, the commuting pairs
in the centralizer of $C$ correspond to the closed points of
$\js$ as $F$ is algebraically closed. Irreducibility of $\js$ would show that for all commuting
triples of $n(k+1) \times n(k+1)$ matrices $(A,B,C)$ with $C$
having Jordan form equal to $n$ copies of $J_{k+1}$ upto addition of scalars, the dimension
of $F[A,B,C]$ would be bounded by $n(k+1)$. (See Corollary \ref{cor: algebra dimension 3(k+1)} ahead for the $n=3$ case, for example.) It is thus of interest
to know if $\js$ is irreducible.

The goal of this paper is to show that the answer to the question
is negative in general: we prove that for $k\ge 1$, there exists
an integer $N(k)$ such that for all $n\ge N(k)$, $\js$ is
reducible.  At the other end of the spectrum, we also show that
$\js$ is irreducible for $n=3$ and all $k\ge 1$. (It is known--see
\cite{NS}--that $\js$ is irreducible for $n=2$ and all $k\ge 1$. )

We note in passing that the results of \cite{NS} showed that for
$n=2$, $\js$ is very naturally related to jet schemes over
determinantal varieties of rank at most $1$, which then led the
first named author along with Toma\v{z} Ko\v{s}ir to study jet
schemes over determinantal varieties in general (\cite{KoSe1},
\cite{KoSe2}, see also \cite{J}).  These were also studied
independently by Cornelia Yuen in her thesis (\cite{CY}).

This paper is a direct result of collaboration between the authors
during the 5th Linear Algebra Workshop held at Kranjska Gora,
Slovenia in 2008.  The authors wish to thank the organizers for
the enjoyable and very productive meeting.

\section  {A Distinguished Open Set of
$\js$}\label{secn:openset} Consider the open subscheme $\U$
defined by the condition that $A_0 = (x\mysup{0}_{i,j})$ is $1$-regular. (Recall that
the matrix $A_0$ is $r$-regular if all eigenspaces of $A_0$ are of
dimension at most $r$; this is an open set condition. See
\cite[Prop. 1]{NS}. $1$-regular matrices in $M_n(F)$ can be
characterized by several equivalent conditions, including: (i)
$\text{dim}_F(F[A_0]) = n$, and (ii) the centralizer of $F[A_0]$
in $M_n(F)$ is precisely $F[A_0]$.)  The goal of this section is
to prove that $\U$ is irreducible, of dimension $ (n^2+n)(k+1)$.

We will give two proofs of this result. The first proof involves a
characterization of matrices $A(t)=A_0 + A_1 t + \cdots + A_k t^k$
such that $A_0$ is $1$-regular in terms of the algebra generated
by $A(t)$ and $t$.
The second proof invokes general results about jet schemes over
smooth schemes and uses a Jacobian computation in \cite{NeSa}.
Once $\U$ is known to be irreducible, we can argue that if $\js$
were irreducible it should equal the closure of $\U$, so it should
also have dimension $ (n^2+n)(k+1)$. In Section \ref{secn:red}
ahead, we will use this expected dimension to show that $\js$ is
reducible for large enough $n$.

\subsection{A Characterization of Matrices $A(t)$ with
$A_0$ $1$-regular} \label{subsec:characterization of A(t)}
We begin with the first proof that $\U$ is irreducible of the
stated dimension.
Let $\cjs$ denote  the
closed points of $\js$, which, since $F$ is algebraically closed, is just the algebraic set in
$\textbf{A}^{2n^2(k+1)}$ consisting of matrices $A(t)=A_0 +
A_1 t + \cdots + A_k t^k$, and $B(t)=B_0 + B_1 t + \cdots + B_k
t^k$ that commute in $M_n(F)[t]/t^{k+1}$. Similarly, let $U = \U \cap \cjs$  be the set of
closed points of $\U$ (\cite[Chap. 2, Lemma 4.3]{Liu}), so $U$ consists of those commuting pairs above
in which $A_0$ is $1$-regular.  Note that $\cjs$ is dense in $\js$ (see e.g. \cite[Chap. 2, Remark 3.49]{Liu}).  Using the fact that every prime ideal of $R$ (where $R$ is as in Section \ref{secn:intro}) is an intersection of maximal ideals, it is easy to see that $U$ is dense in $\U$ and that the dimension of $U$ as an algebraic set is the same as the dimension of $\U$.  We will hence
work with $\cjs$ and $U$.  The goal in this subsection is to
prove the irreducibility of $U$ by deriving  a characterization of
matrices $A(t)=A_0 + A_1 t + \cdots + A_k t^k$ such that $A_0$ is
$1$-regular in terms of the algebra generated by $A(t)$ and $t$.

\begin{Notation} \label{notn:symm_poly} Let $z_1$, $\dots$, $z_k$ be
noncommuting variables.  Given $n_1$ copies of $z_1$, $\dots$,
$n_k$ copies of $z_k$,   we will denote by
$S(z_1^{[n_1]},\dots,z_k^{[n_k]})$ the sum of all noncommuting
monomials of degree $n_1$ in $z_1$, $\dots$, $n_k$ in $z_k$.  For
instance, $S(x^{[2]}, y^{[1]}) = x^2y + xyx + yx^2$.

We will find it convenient to also adopt the following convention:
given the not necessarily   distinct noncommutative variables
$y_{1}$, $\dots$, $y_{t}$, of which $n_1$ are all equal (to say
$y_{i_1}$), $\dots$, $n_k$ are all equal (to say $y_{i_k}$), $n_1
+ \cdots + n_k = t$, the expression   $S(y_{1},\dots , y_{t})$
will denote the polynomial
$S(y_{i_1}^{[n_1]},\dots,y_{i_k}^{[n_k]})$ defined above.  Notice
that if $y_{\sigma(1)}$, $\dots$, $y_{\sigma(t)}$ is any
rearrangement of the list $y_{1}$, $\dots$, $y_{t}$, then
$S(y_{1},\dots , y_{t}) = S(y_{\sigma(1)},\dots , y_{\sigma(t)})$
just by definition.

Expressions such as $S(x^{[d]},y_{1},\dots , y_{t})$, with $d>0$
and $y_{1}$, $\dots$, $y_{t}$ not necessarily distinct among
themselves but distinct from $x$, are similarly defined.  We
define $S(x^{[0]},y_{1},\dots , y_{t})$ to be $S(y_{1},\dots ,
y_{t})$, and $S(x^{[d]},y_{1},\dots , y_{t})$ to be zero if $d<0$.
\end{Notation}

\begin{Notation} \label{notn:dy(q)} Let $q(x) = \sum_{j=0}^{N-1} c_j
x^j$ be a polynomial with coefficients in $F$ in the variable $x$.
Given the not necessarily distinct noncommutative variables
$y_{1}$, $\dots$, $y_{t}$, none of which commutes with $x$, we let
$d_{y_{1},\dots,y_{{t}}}(q(x))$ denote the (noncommutative)
polynomial  $\sum_{j=0}^{N-1} c_j
S(x^{[j-t]},y_{1},\dots,y_{{t}})$ (recall from above the
convention for $S(x^{[j-t]},y_{1},\dots,y_{{t}})$ if $j-t \le 0$).
%
Note that the polynomial
$d_{y_{1},\dots,y_{{t}}}(q(x))$ is   symmetric with respect to the
permutation of the $y_{j}$ because $S$ is symmetric with respect
to the permutation of the $y_{j}$.

Given matrices $A_0$, $\dots$, $A_k$, an expression such as
$d_{A_{i_1}\dots A_{i_r}}(q_j(A_0))$ indicates the following: Take
the polynomial $d_{y_{i_1}\dots y_{i_r}}(q_j(x))$ defined above
for the noncommuting variables $x, y_{i_1},\dots, y_{i_r}$ and the
polynomial $q_j(x) \in F[x]$, and evaluate it at $x=A_0$ and
$y_{i_j} = A_{i_j}$.
\end{Notation}

We begin with the following lemma which will go into our
characterization:

\begin{Lemma} \label{lemma:equivalent_powers_of_A(t)} Given $N > 0$
and $N(k+1)$ elements $c_{i,j} \in F$ for $0\le i < N$, $0 \le j
\le k$, define the polynomials $q_j(x)$, $0\le j \le k$ by $q_j(x)
= \sum_{i=0}^{N-1} c_{i,j} x^i$.  Given $A(t) = A_0 + A_1 t +
\cdots + A_k t^k$ in $M_n(F)[t]/t^{k+1}$, we have
\begin{equation} \label{eqn:rewrite_powers_of_A(t)}
\sum_{j=0}^k \sum_{i=0}^{N-1} c_{i,j}\left(A(t)\right)^i t^j = B_0
+ B_1 t + \cdots + B_k t^k
\end{equation}
 where $B_0 = q_0(A_0)$, and more
generally, for $s = 1,\dots, k$,
\begin{equation} \label{eqn:formula_for_B_s}
B_s = \sum_{j=0}^{s-1}\left( \sum_{r=1}^{s-j} \sum_{{i_1\ge i_2
\ge \dots \ge i_r
> 0 }\atop{ i_1+ \cdots + i_r = s-j}}d_{A_{i_1}\dots A_{i_r}}(q_j(A_0)) \right) + q_s(A_0)
\end{equation}

\end{Lemma}

\begin{proof} We first consider a single power
$\left(A(t)\right)^i = \left(A_0 +A_1 t + \cdots + A_k
t^k\right)^i$.  Expanding the parenthesis, we may write this as
$\sum_{l=0}^k G_l^{(i)} t^l$, where $G_0^{(i)} =
\left(A_0\right)^i$, and for $l \ge 1$,
\begin{equation} \label{eqn:formula_for_G_l^{(i)}}
G_l^{(i)} = \sum_{r=1}^l \sum_{{i_1\ge i_2 \ge \dots \ge i_r
> 0 }\atop{ i_1+ \cdots + i_r = l}} S(A_0^{[i-r]},A_{i_1},\dots
A_{i_r})
\end{equation}

We thus write \begin{equation} \label{eqn:rewrite_in_terms_of_G}
\sum_{j=0}^k \sum_{i=0}^{N-1} c_{i,j}\left(A(t)\right)^i t^j =
\sum_{j=0}^k \sum_{i=0}^{N-1} c_{i,j}(G_0^{(i)} + G_1^{(i)} t +
\cdots + G_k^{(i)} t^k) t^j
\end{equation}
We wish to rewrite the right side of the equation above as $B_0 +
B_1 t + \cdots + B_k t^k$. Now $B_0$ is the coefficient of $t^0$
on the right side of the equation above, which is simply
$\sum_{i=0}^{N-1} c_{i,0} G_0^{(i)}$. Since $G_0^{(i)} =
\left(A_0\right)^i$, we find $B_0 = q_0(A_0)$, with $q_0$ as in
the statement of the theorem.

Similarly, for $s \ge 1$, $B_s$ is the coefficient of $t^s$ on the
right side of Equation (\ref{eqn:rewrite_in_terms_of_G}), so
\begin{equation} \label{eqn:collect_B_s}
B_s = \sum_{j=0}^s \sum_{i=0}^{N-1} c_{i,j} G_{s-j}^{(i)}
\end{equation}
When $j=s$, the inner term on the right side of the equation above
is $\sum_{i=0}^{N-1} c_{i,s} G_{0}^{(i)} = \sum_{i=0}^{N-1}
c_{i,s} \left(A_0\right)^i = q_s(A_0) $.

For a fixed $j = j^* < s$, the the inner term on the right side of
Equation (\ref{eqn:collect_B_s}) equals, from Equation
(\ref{eqn:formula_for_G_l^{(i)}}) above,
\begin{eqnarray*}
\sum_{i=0}^{N-1} c_{i,j^*}\left[ \sum_{r=1}^{s-j^*} \sum_{{i_1\ge
i_2 \ge \dots \ge i_r
> 0 }\atop{ i_1+ \cdots + i_r =
 s-j^*}} S(A_0^{[i-r]},A_{i_1},\dots
A_{i_r}) \right] &=&\\
 \sum_{r=1}^{s-j^*} \sum_{{i_1\ge i_2 \ge
\dots \ge i_r
> 0 }\atop{ i_1+ \cdots + i_r = s-j^*}} d_{A_{i_1}\dots A_{i_r}}(
q_{j^*}(A_0)) &&
\end{eqnarray*}
Adding together these expressions for $j = s, s-1, \dots, 0$, we
find that $B_s$ is indeed as described in the statement of the
theorem.

\end{proof}

We now prove the following:

\begin{Theorem} \label{theorem:characterization of A(t) with A0
1-regular}  Let $A=A(t) = A_0 + A_1 t + \cdots + A_k t^k$ be given
in $M_n(F)[t]/t^{k+1} \subset M_{n(k+1)}(F)$.  Then the following
are equivalent:
\begin{enumerate}
\item \label{theorem_enum:1-regular} $A_0$ is $1$-regular.
\item\label{theorem_enum:linear_independent_powers} The $n(k+1)$
elements $\left(A(t)\right)^i t^j $, $0 \le i < n$, $0 \le j \le
k$ in $F[A,t]$ are $F$-linearly independent.
\item\label{theorem_enum:F[A,t] has full dimension} $F[A,t]$
has dimension $n(k+1)$ as an  $F$-subalgebra of
$M_n(F)[t]/t^{k+1}$ (and hence of $M_{n(k+1)}(F)$).
\item\label{theorem:enum: B(t) is poly in A and t} Any $B=B(t)$ that commutes with both $A(t)$ and $t$ is
for the form $B(t) = \sum_{j=0}^k \sum_{i=0}^{n-1}
c_{i,j}\left(A(t)\right)^i t^j$, for arbitrary elements $c_{i,j}
\in F$, $0 \le i < n$, $0 \le j \le k$.
\item\label{theorem:enum: inductive_form_of_B(t)} Any $B=B(t)$ that commutes with both $A(t)$ and $t$ is
for the form $B(t) = B_0 + B_1 t + \cdots + B_kt^k$, where $B_0 =
q_0(A_0)$, and more generally, for $s = 1,\dots, k$,
\begin{equation*} \label{theorem:enum:formula_for_B_s}
B_s = \sum_{j=0}^{s-1}\left( \sum_{r=1}^{s-j} \sum_{{i_1\ge i_2
\ge \dots \ge i_r
> 0 }\atop{ i_1+ \cdots + i_r = s-j}}d_{A_{i_1}\dots A_{i_r}}(q_j(A_0)) \right) + q_s(A_0)
\end{equation*} where $q_0$, $q_1$, $\dots$, $q_k$ are arbitrary
polynomials with coefficients in $F$ of degree at most $n-1$.

\end{enumerate}

\end{Theorem}
\begin{proof}
Note that (\ref{theorem:enum: B(t) is poly in A and t})
$\Leftrightarrow$ (\ref{theorem:enum: inductive_form_of_B(t)})
 is an
immediate consequence of Lemma
\ref{lemma:equivalent_powers_of_A(t)}, with the $N$ of that lemma
replaced by $n$.

\noindent (\ref{theorem_enum:1-regular}) $\Rightarrow$
(\ref{theorem_enum:linear_independent_powers}):   Since $A_0$ is
$1$-regular, the matrices $I_n$, $A_0$, $\dots$,
$\left(A_0\right)^{n-1}$ are $F$-linearly independent in $M_n(F)$.
It follows immediately, by sequentially considering the
coefficients of $t^0$, $t^1$, $\dots$, $t^k$ in the left hand side
of an equation such as
 \begin{equation*} \sum_{j=0}^k
\sum_{i=0}^{n-1} c_{i,j}\left(A(t)\right)^i t^j = 0
\end{equation*}
 for arbitrary elements $c_{i,j} \in F$, that the $n(k+1)$
elements $\left(A(t)\right)^i t^j $, $0 \le i < n$, $0 \le j \le
k$, must be $F$-linearly independent.

\noindent (\ref{theorem_enum:linear_independent_powers})
$\Rightarrow$ (\ref{theorem_enum:F[A,t] has full dimension}): This
implication is clear since the algebra generated by two commuting
matrices in $M_{n(k+1)}(F)$ has dimension at most $n(k+1)$ by
classical results (see \cite{Gur} for example).

\noindent (\ref{theorem_enum:F[A,t] has full dimension})
$\Rightarrow$ (\ref{theorem_enum:1-regular}): Let ${A'} =
A_0 + A_1 t + \cdots + A_k t^k + 0 t^{k+1} + \cdots + 0 t^{k+l}
\in M_n(F)[t]/t^{k+l+1}$, where $l \ge 0$ is yet to be determined.
(Note that $M_n(F)[t]/t^{k+l+1} \subset M_{n(k+l+1)}(F)$, with $t$
corresponding to the $n(k+l+1) \times n(k+l+1)$ matrix consisting
of $n$ Jordan blocks of size $k+l+1$ along the diagonal.) Write
$E$ for the algebra $F[A,t] \subset M_n(F)[t]/t^{k+1}$, and
${E'}$ for the algebra $F[{A'},t] \subset
M_n(F)[t]/t^{k+l+1}$. Given a nonzero element $B = B_i t^i +
B_{i+1}t^{i+1} + \cdots$ in $M_n(F)[t]/t^{k+1}$ with $B_i \neq 0$,
we define the \textsl{degree} of $B$ to be $i$; we also define the
degree of $0$ to be infinity. We have the filtration by $F$-spaces
$E = E_0 \supset E_1 \supset \cdots \supset E_k \supset
E_{k+1}=0$, where $E_i$ consists of those elements in $E$ with
degree $i$ or higher. We have a similar filtration ${E'}
= {E'_0} \supset {E'_1} \supset \cdots \supset
{E'_{k+l}} \supset  {E'_{k+l+1}} = 0$.
 The map $\pi:
M_n(F)[t]/t^{k+l+1} \rightarrow M_n(F)[t]/t^{k+1}$ that maps the
element $B_0 + B_1 t + \cdots + B_k t^k + B_{k+1} t^{k+1} + \cdots
+ B_{k+l}t^{k+l}$ to $B_0 + B_1 t + \cdots + B_k t^k$ is an
$F$-algebra homomorphism and in particular an $F$ vector space
homomorphism. Equation (\ref{eqn:formula_for_B_s}) in Lemma
\ref{lemma:equivalent_powers_of_A(t)} (applied with $k+l$
substituted for $k$) shows that the coefficient of $t^s$ of a
polynomial expression $f({A'},t)$ depends only on the coefficients
of $t^0$, $t^1$, $\dots$, $t^s$ in ${A'}$. More specifically, if
$f({A'},t) =  B_it^i + B_{i+1} t^{i+1} + \cdots + B_k t^k +
B_{k+1} t^{k+1} + \cdots + B_{k+l}t^{k+l}$ $\mod t^{k+l+1}$, then
$f(A,t) = B_it^i + B_{i+1} t^{i+1} + \cdots + B_k t^k$ $\mod
t^{k+1}$.  Thus, $\pi(f({A'},t)) = f(A,t)$, so $\pi$ maps each
${E'_i}$ surjectively to ${E_i}$ for $i = 0, \dots, k$. Moreover,
the kernel of the induced  $F$ vector space map $ {E'_i} \mapsto
E_i \mapsto E_i/E_{i+1}$ is precisely $ {E'_{i+1}}$.  We thus have
$F$ vector space isomporphisms between $ {E'_i}/{E'_{i+1}}$ and
$E_i/E_{i+1}$ for $i=0,\dots, k$.

Now assume that $F[A,t]$ has dimension $n(k+1)$ but that $A_0$ is
not $1$-regular.  If we write a polynomial expression $f(A,t)$ as
$B_0 + B_1 t + \cdots + B_k t^k$, Lemma
\ref{lemma:equivalent_powers_of_A(t)} tells us that $B_0 =
q_0(A_0)$, for a suitable polynomial $q_0$. Conversely, given a
polynomial $q_0 \in F[x]$, $q_0(A) = q_0(A_0) + t(\ldots)$. It
follows that $E_0/E_1 \cong F[A_0]$.   Since $A_0$   not being
$1$-regular means that $\dim_F F[A_0] < n$, we find $\dim_F
E_0/E_1 < n$. The fact that $\dim_F E = n(k+1)$ and that $\dim_F E
= \sum_{i=0}^k \dim_F \left( E_i/E_{i+1}\right)$ shows that
$\dim_F \left( E_i/E_{i+1}\right)
> n$ for some $i$ with $k\ge i\ge 1$.  It follows from the isomorphism above
 that $\dim_F {E'_i}/{E'_{i+1}} > n$.  But for any $j < k+l$, we
have an injective $F$ vector space map from
${E'_j}/{E'_{j+1}}$ to
${E'_{j+1}}/{E'_{j+2}}$ that sends the class of
an element $B_jt^j + B_{j+1}t^{j+1} + \cdots$ to the class of the
element $B_jt^{j+1} + B_{j+1}t^{j+2} + \cdots$.
Hence $\dim_F {E'_j}/{E'_{j+1}} > n$ for all $j$
with $k+l \ge j \ge i$. It is clear that by taking $l$ large
enough, we can make $\dim_F {E'} = \sum_{i=0}^{k+l}
\dim_F {E'_i}/{E'_{i+1}}
> n(k+l+1)$, no matter what the values of $\dim_F {E'_s}/{E'_{s+1}}$
are for $s < i$. But this violates the classical result that the
dimension of the $F$-algebra generated by the commuting matrices
$ A'$ and $t$ in $M_n(F)[t]/t^{k+l+1} \subset
M_{n(k+l+1)}(F)$ is bounded by $n(k+l+1)$.  Hence, $A_0$ must be
$1$-regular.

\noindent (\ref{theorem_enum:F[A,t] has full dimension})
$\Leftrightarrow$ (\ref{theorem:enum: B(t) is poly in A and t}):
Both implications follow from  \cite[Theorem 1.1]{NeSa}, which
states that $F[A,t] \subseteq M_{n(k+1)}(F)$ has dimension
$n(k+1)$ iff the centralizer of $F[A,t]$ in $M_{n(k+1)}(F)$ is
$F[A,t]$ itself, along with Part
(\ref{theorem_enum:linear_independent_powers}) above of this
theorem.

\end{proof}

We have the following corollary immediately:

\begin{Corollary} \label{cor:irred of U}
The sets $U \subset \cjs$ and $\U\subset \js$ are irreducible, of
dimension $(n^2+n)(k+1)$.
\end{Corollary}
\begin{proof}
Viewed as algebraic sets, the set of all $A(t)$ with $A_0$ $1$-regular is an open set in
$\mathbf{A}^{n^2(k+1)}$, since $A_0$ is constrained to live in the
open set of $\mathbf{A}^{n^2}$ of $1$-regular matrices while
$A_1$, $\dots$, $A_k$ can be arbitrary.  Theorem
\ref{theorem:characterization of A(t) with A0 1-regular},  Part
(\ref{theorem:enum: B(t) is poly in A and t})  above shows that
the set of matrices $B(t)$ that commute with $A(t)$ is defined by
the $n(k+1)$ arbitrary elements $c_{i,j} \in F$. Hence  $U$, which is the set of commuting pairs $(A(t), B(t))$ with $A_0$ $1$-regular, is
isomorphic as an algebraic set to a product of an open set of $\mathbf{A}^{n^2(k+1)}$
and $\mathbf{A}^{n(k+1)}$. The irreducibility and
dimension of $U$ immediately follow.  The same results hold for $\U$ as described at the beginning of this subsection.


\end{proof}

The theorem yields another corollary:

\begin{Corollary} \label{cor:surectivity of projection}
Let $A(t) = A_0 + A_1 t + \cdots + A_k t^k $ and $B(t) = B_0 + B_1
t + \cdots + B_k t^k$ commute in $M_n(F)[t]/t^{k+1}$, and assume
that $A_0$ is $1$-regular. Then for any $A_{k+1} \in M_n(F)$,
there exists $B_{k+1} \in M_n(F)$ such that $A'(t) = A_0 + A_1 t +
\cdots + A_k t^k + A_{k+1}t_{k+1}$ and $B(t) = B_0 + B_1 t +
\cdots + B_k t^k + B_{k+1}t_{k+1}$ commute in $M_n(F)[t]/t^{k+2}$.
In particular, the map $\pi_{k+1} :
\mathcal{V}(\mathcal{C}\mysup{k+1}_{2,n})\rightarrow \cjs$ that
sends a general pair $(A_0 + A_1 t + \cdots + A_k t^k +
A_{k+1}t_{k+1},B_0 + B_1 t + \cdots + B_k t^k + B_{k+1}t_{k+1})$
to $(A_0 + A_1 t + \cdots + A_k t^k , B_0 + B_1 t + \cdots + B_k
t^k )$ is surjective when restricted to the open sets of
$\mathcal{V}(\mathcal{C}\mysup{k+1}_{2,n})$ and $\cjs$ where $A_0$
is $1$-regular.

\end{Corollary}
\begin{proof}
This is immediate from Theorem \ref{theorem:characterization of
A(t) with A0 1-regular},  Part (\ref{theorem:enum:
inductive_form_of_B(t)}), where we see that the solution for $B_s$
$(s=0, \dots, k)$ only depends on $A_i$ $(i=0, \dots, s)$,
and that having solved for $B_s$, one
can solve for $B_{s+1}$ given any $A_{s+1}$.
\end{proof}

\begin{Remark}  Note that $\pi_1: \mathcal{C}\mysup{1}_{2,n}
\rightarrow \mathcal{C}_{2,n}$ is trivially surjective since given
$A_0$ and $B_0$ in $M_n(F)$ that commute, the matrices $A_0 +
0\cdot t$ and $B_0 + 0\cdot t$ commute in $M_n(F)[t]/t^2$.
However, this simple extension to higher degree in $t$ fails for
$k
> 1$. For instance, let $X$ and $Y$ be any two matrices in $M_n(F)$
that do not commute with each other.  Then $A = 0 + X t$ and $B =
0 + Yt$ commute in $M_n(F)[t]/t^2$, but it is clear by writing
down the equation for the $t^2$ component that $A' = 0 + X t + P
t^2$ and $B' = 0 + Yt + Q t^2$ cannot commute in $M_n(F)[t]/t^3$
for any choice of $P$ and $Q$.

\end{Remark}

We also have the following:
\begin{Corollary} \label{cor:stuff_in_ad(A_0)}
Given the commuting pair $A(t) = A_0 + A_1 t + \cdots + A_k t^k $
and $B(t) = B_0 + B_1 t + \cdots + B_k t^k$   in
$M_n(F)[t]/t^{k+1}$, with $A_0$ $1$-regular, we have:
\begin{equation*}
\left[A_1,B_k\right] + \left[A_2,B_{k-1}\right] + \cdots +
\left[A_k,B_1\right] \in Ad(A_0)
\end{equation*}

\end{Corollary}
\begin{proof}
Taking $A_{k+1}=0$, Corollary \ref{cor:surectivity of projection}
shows that we can find $B_{k+1}$ such that $A_0 + A_1 t + \cdots +
A_k t^k +  0\cdot t^{k+1}$ and $B(t) = B_0 + B_1 t + \cdots + B_k
t^k + B_{k+1}t^{k+1}$ commute in $M_n(F)[t]/t^{k+2}$. Writing down
the equation for the matrices in $t^{k+1}$, we find
\begin{equation*}
\left[A_0, B_{k+1}\right] +  \left[A_1,B_k\right] +
\left[A_2,B_{k-1}\right] + \cdots + \left[A_k,B_1\right]   = 0
\end{equation*}
Since the first term is in $Ad(A_0)$, the sum of the remaining
terms must also be so.

\end{proof}

\begin{Remark} Note that the proofs of Theorem \ref{theorem:characterization of A(t) with A0
1-regular}, and Corollaries \ref{cor:surectivity of projection} and \ref{cor:stuff_in_ad(A_0)} did not depend on $F$ being algebraically closed.
\end{Remark}

\subsection{Alternative Proof of Corollary \ref{cor:irred of U}}
For the second proof that $\U$ is irreducible of dimension
$(n^2+n)(k+1)$, we recall the following: if $X$ is a smooth scheme
of finite type over $F$ of dimension $d$, then $X\mysup{k}$ is an
$\textbf{A}^{dk}$ bundle over $X$.  See \cite[\S 2]{M2} for
instance. The heart of the result is that for a smooth scheme of
finite type over $F$ of dimension $d$, there exists an open cover
$U_{\alpha}$ with \'{e}tale maps $g_{\alpha}: U_{\alpha}
\rightarrow \textbf{A}^d$ (see e.g., \cite[\S 6.2.2]{Liu}), and
since $g_{\alpha}$ is \'{e}tale, $U_{\alpha}\mysup{k} \cong
\left(\textbf{A}^d\right)\mysup{k} \times_{\textbf{A}^d}
U_{\alpha}$.

Recall that since $F$ is algebraically closed, smoothness and regularity coincide (see \cite[Chap. 4, Definition 3.28]{Liu}).
In \cite[Theorem 1.1]{NeSa}, the authors use the Jacobian criterion
to show that a closed point
$(A_0,B_0)$ of the commuting matrices scheme $\mathcal{C}_{2,n}$
is regular  if and only if
$\text{dim}_F[A_0,B_0]=n$. (Note that the Jacobian criterion in
\cite{NeSa} is applied to the ideal $J$ defined in Section
\ref{secn:intro}, which is not known to be radical, and hence the
result of \cite[Theorem 1.1, Part 4]{NeSa} really applies to the
commuting matrices \textit{scheme} and not necessarily the
commuting matrices variety in the sense used in \cite{NeSa}.) When
$A_0$ is $1$-regular, then the two characterizations described at
the beginning of Section \ref{secn:openset} show that
$\text{dim}_F[A_0,B_0]=n$ for any $B_0$ that commutes with $A_0$.
Thus, letting $\mathcal{W}$ denote the open subscheme of
$\mathcal{C}_{2,n}$ where $A_0$ is $1$-regular and $W$ the set of
closed points of $\mathcal{W}$, we find that the points of $W$ are
all nonsingular, and hence (see \cite[Chap. 4, Corollary 2.17]{Liu} for
instance), $\mathcal{W}$ is smooth. Thus, $\U$, which is
$\mathcal{W}\mysup{k}$ (as $\mathcal{W} \rightarrow \mathcal{C}_{2,n}$ is \'{e}tale) is
simply an $\mathbf{A}^{(n^2+n)k}$ bundle over
$\mathcal{W}$. Being an open set of $\mathcal{C}_{2,n}$,
$\mathcal{W}$ is irreducible of dimension $n^2+n$, so $\U$ is
irreducible of dimension $(n^2+n)(k+1)$.

\section {Reducibility of $\js$ for Large
$n$} \label{secn:red}

We prove in this section that for any $k$ and large enough $n$,
$\js$ is reducible. We show that for each $k\ge 1$, there exists
an integer $N(k)$ such that for all $n\ge N(k)$, $\cjs$ has a
closed set of dimension greater than $(n^2+n)(k+1)$.  Since the
open set $U$ considered in the previous section is irreducible of
dimension $(n^2+n)(k+1)$, $\cjs$ cannot be irreducible for such
$n$, and hence $\js$ cannot be irreducible for such $n$.


\begin{Theorem}
For each $k\ge 1$, there exists an integer $N(k)$ such that for
all $n\ge N(k)$, $\cjs$, and hence $\js$, is reducible.
\end{Theorem}

\begin{proof}

Let $n=3a+b$, and write $n\times n$ matrices as $4\times 4$ block
matrices where the first 3 rows and columns are of dimension $a$
and the last ones are of dimension $b$.  Let $W$ be the closed set
of all pairs $(A(t),B(t))\in \cjs$ such that in block form
$$A_0=\left(
\begin{array}{cccc}0&I&0&0\\0&0&I&0\\0&0&0&0\\0&0&0&0
\end{array}
\right), B_0=\left(
\begin{array}{cccc}0&B_1^{\mysup 0}&B_2^{\mysup 0}&B_3^{\mysup 0}\\0&0&B_1^{\mysup 0}&0\\0&0&0&0\\
0&0&B_4^{\mysup 0}&0
\end{array}
\right),$$
$$A_1=\left(
\begin{array}{cccc}A_{11}^{\mysup 1}&A_{12}^{\mysup 1}&A_{13}^{\mysup 1}&A_{14}^{\mysup 1}\\
A_{21}^{\mysup 1}&A_{22}^{\mysup 1}&A_{23}^{\mysup 1}&A_{24}^{\mysup 1}\\
0&A_{32}^{\mysup 1}&A_{33}^{\mysup 1}&A_{34}^{\mysup
1}\\A_{41}^{\mysup 1}&A_{42}^{\mysup 1}&A_{43}^{\mysup
1}&A_{44}^{\mysup 1}
\end{array}
\right),B_1=\left(
\begin{array}{cccc}B_{11}^{\mysup 1}&B_{12}^{\mysup 1}&B_{13}^{\mysup 1}&B_{14}^{\mysup 1}\\
B_{21}^{\mysup 1}&B_{22}^{\mysup 1}&B_{23}^{\mysup 1}&B_{24}^{\mysup 1}\\
0&B_{32}^{\mysup 1}&B_{33}^{\mysup 1}&B_{34}^{\mysup
1}\\B_{41}^{\mysup 1}&B_{42}^{\mysup 1}&B_{43}^{\mysup
1}&B_{44}^{\mysup 1}
\end{array}
\right)$$ and $A_2,\ldots ,A_k,B_2,\ldots ,B_k\in M_n(F)$ are
arbitrary such that $[A(t),B(t)]=0$. The commutativity relation of
$A(t)$ and $B(t)$ is given by.

\begin{eqnarray}
\nonumber\left[A_0,B_0\right] &=& 0\\
\nonumber\left[A_0,B_1\right]+\left[A_1,B_0\right]&=& 0\\
\left[A_0,B_2\right]+\left[A_1,B_1\right]+\left[A_2,B_0\right]&=&0\label{eqn:SystemForB}\\
\nonumber\vdots &=& \vdots\\
\nonumber\left[A_0,B_k\right]+\left[A_1,B_{k-1}\right]+\cdots
+\left[A_k,B_0\right]&=&0
\end{eqnarray}

The first equation is automatically satisfied since $B_0$ has been
chosen to commute with $A_0$.  As for the second equation, note
that $\left[A_0,B_1\right] + \left[A_1,B_0\right]$ has zeros in
the entries $(2,1)$, $(3,1)$, $(4,1)$, $(3,2)$, and $(3,4)$, so
the second equation in the system of equations
(\ref{eqn:SystemForB}) can be described by $6a^2+4ab+b^2-1$
equations (the subtraction of $1$ is because the trace of $[X,Y]$
is zero, so the diagonal entries are always dependent).  Each of
the remaining equations is given by $n^2 -1$ equations. Therefore
$W$ has dimension at least
%
\begin{eqnarray*}
2a^2+2ab+2(8a^2+6ab+b^2) &+& 2(k-1)n^2 -(6a^2+4ab+b^2-1)\\
-(k-1)n^2+k-1 &=&12a^2+10ab+b^2+(k-1)n^2+k.
\end{eqnarray*}

Note that $Gl_n(F)$ acts on $\cjs$ by simultaneous conjugation on
each component $A_i$ and $B_j$. Let $V$ be the set of all
$(A',B')\in \cjs$ such that $A_0'$ is similar to $\lambda I+A_0$
for some $\lambda \in F$. Then $V$ contains the set $S= \{(\lambda
+ gAg^{-1}, \mu + gBg^{-1})\ |\ (A,B)\in W, \ \lambda, \mu \in F,\
g\in Gl_n(F)\}$. Let us denote by $C(A_0)$ the centralizer of
$A_0$ in $Gl_n(F)$. Then $C(A_0)$ consists of all invertible
matrices of the block form
\begin{eqnarray*}
\left(
\begin{array}{cccc}X&Y&Z&P\\0&X&Y&0\\0&0&X&0\\
0&0&Q&R
\end{array}
\right)
\end{eqnarray*}
If $\widetilde{W}$ is an irreducible component of $W$,
we have a map $f_{\widetilde{W}}\colon Gl_n(F) \times
\widetilde{W} \times F^2 \rightarrow S\subset V\subset \cjs$ that
takes $g, (A,B), \lambda, \mu$ to $(\lambda + gAg^{-1}, \mu +
gBg^{-1})$.  For any fixed $(\lambda' + g'A'g'^{-1}, \mu' +
g'B'g'^{-1})$ in the image, the fiber over this point is
parameterized by the set $g'C(A_0)$, and hence is of dimension
equal to  $\dim C(A_0)$. From this it follows that the image of
$f_{\widetilde{W}}$ has dimension $n^2-\dim C(A_0)+\dim
{\widetilde{W}}+2$, and hence the set
%
%
$S$, which is the union of the images of $f_{\widetilde{W}}$ as
${\widetilde{W}}$ ranges through the components of $W$, has
dimension   $n^2-\dim C(A_0)+\dim W+2$.  Thus,
\begin{eqnarray*}
\dim V&\ge& \dim(S)= n^2-\dim C(A_0)+\dim W+2\\
&\ge&(3a+b)^2-(3a^2+2ab+b^2)+12a^2+10ab+b^2+(k-1)n^2+k+2\\
&=&18a^2+14ab+b^2+k+2+(k-1)n^2 \end{eqnarray*} However, if $\cjs$
were irreducible, the  expected dimension from Section
\ref{secn:openset}  is $(k+1)(n^2+n)$, which we write as
$$2(n^2+n) + (k-1)(n^2+n) =
18a^2+12ab+2b^2+6a+2b+(k-1)n^2+(k-1)(3a+b)$$  Note that
$\overline{V}$ is proper subvariety of $\cjs$, since, for instance, the coefficient of $t^0$
of every element in $\overline{V}$ lies in the closed set of $M_n(F)$ where matrices are not $1$-regular.
Hence, it
suffices to find $a$ and $b$ such that
$$18a^2+14ab+b^2+k+2+(k-1)n^2\ge 18a^2+12ab+2b^2+6a+2b+(k-1)n^2+(k-1)(3a+b).$$
This is equivalent to $b^2+(k+1-2a)b+(k+1)3a-k-2\le 0$. The
discriminant of this quadratic polynomial $\disc=
(k+1-2a)^2-4\left((k+1)3a-k-2)\right)=4a^2-16(k+1)a+(k+1)^2+4(k+2)
$ must be nonnegative, and
\begin{equation} \label{eqn:ineq for b}
\frac{2a-k-1-\sqrt{\disc}}{2}\le b\le
\frac{2a-k-1+\sqrt{\disc}}{2}
\end{equation}
 Since
$\disc\ge 0$ we get that $a\ge 2(k+1)+
\frac{\sqrt{15(k+1)^2-4(k+2)}}{2}$ or $a\le
2(k+1)-\frac{\sqrt{15(k+1)^2-4(k+2)}}{2}$. Since
$\sqrt{15(k+1)^2-4(k+2)} > 3(k+1)$ for $k\ge 1$, we find in the
second case $2a<k+1$. But the definition of $\disc$ above shows
that $(k+1-2a)^2
> \disc$, so we find
\begin{equation} b\le
\frac{2a-k-1+\sqrt{\disc}}{2} <\frac{2a-k-1+|2a-k-1|}{2}=0
\end{equation}
which is a contradiction. So $a\ge \mu_k$, where we have written
$\mu_k$ for the ceiling function $\left\lceil
2(k+1)+\frac{\sqrt{15(k+1)^2-4(k+2)}}{2}\right\rceil$.
For each $n=3a+b$ such that $a\ge \mu_k$ and $b$ satisfies the
inequality (\ref{eqn:ineq for b}) above
$V$ does not have smaller dimension than the expected dimension of
$\cjs$, so $\cjs$ is reducible.  In particular, for $k=1$, we find
$\mu_k = 8$,  $a$ may be taken as $8$ and $b$ may be taken as $5$,
so $\cjones$ is reducible for $n=29$.

We will now look at the solutions for $n$ for a fixed $k$ and show
that for each $k\in \mathbb{N}$ there exists $N(k)$ such that
$\cjs$ is reducible if $n\ge N(k)$. We write the inequalities in
(\ref{eqn:ineq for b}) as
\begin{equation*}
a - \frac{k+1}{2} -\frac{\sqrt{\disc}}{2} \le b \le a -
\frac{k+1}{2} +\frac{\sqrt{\disc}}{2}
\end{equation*}
Adding $3a$ to all sides, we find
\begin{equation} \label{eqn:equation for n in terms of a}
4a - \frac{k+1}{2} -\frac{\sqrt{\disc}}{2} \le 3a+b=n \le 4a -
\frac{k+1}{2} +\frac{\sqrt{\disc}}{2}
\end{equation}

%


First, we remove the dependency of $\disc$ on $a$ by noting that
if $a\ge \beta_k = \left\lceil
2(k+1)+\frac{\sqrt{15(k+1)^2-4(k+1)+12}}{2} \right\rceil
> \mu_k$, then $\sqrt{\disc}\ge 4$.  Thus, we may narrow the
range for $n$ given by (\ref{eqn:equation for n in terms of a}) by
choosing $a\ge \beta_k > \mu_k$ to find
\begin{equation} \label{eqn:simplified eqn for n in terms of a}
4a- \frac{k+5}{2}\le n\le 4a-\frac{k-3}{2}
\end{equation}
Since the two terms on either side differ by $4$, there are four
integer values starting from the smallest $n = n_a,$ then $n_a+1,
n_a+2, n_a+3$, that satisfy the inequality above and for which
$\cjs$ will be reducible. Moreover, by replacing $a$ by $a+1$, we
find $\cjs$ will be reducible for $n_{a+1} = n_a+4, n_a+5, n_a+6,
n_a+7$, and so on. Hence, by taking $N(k) = \lceil 4\beta_k-
\frac{k+5}{2} \rceil$, we find that $\cjs $ will be reducible for
$n\ge N(k)$, with $n$ written as $3a+b$ and $a$ obtained by
solving (\ref{eqn:simplified eqn for n in terms of a}), along with
the original limit of $a\ge \beta_k$:
\begin{equation} \label{eqn:a in terms of n}
\frac{1}{4}\left(n + \frac{k-3}{2}\right) \le a \le
\frac{1}{4}\left(n + \frac{k+5}{2} \right)\quad \text{and}\ a \ge
\beta_k
\end{equation}
(If $k=1$, we recover our earlier value of $n=29$. If $k=2$, we
find $\beta_k = 12$ and $N(2) = 45$. However, working directly
with (\ref{eqn:ineq for b}), we find $\cjs$ is also irreducible
for $n=44$ when $k=2$.)


\end{proof}

\section {Irreducibility of $\js$ for $n=3$} \label{secn:small_irred}

We will show here that when $n=3$, the set $U$  is dense in
$\cjsthree$ for all $k$. This will immediately show that $\jsthree$ is
irreducible for all $k$. (Recall that $\jstwo$ is already known to be
irreducible for all $k$, see \cite{NS}.)

First, we need some general reduction results that hold for any
$n$. We will denote the closure of $U$ in $\cjs$ by $\Ubar$.

\begin{Lemma} \label{lemma:assume A0B0 one eval}
Assume that for all $m < n$, $\cjsm$ has been proven to be
irreducible.  Then, any point $(A=A(t),B=B(t)) \in \cjs$ where
$A_0$ or $B_0$ have at least two distinct eigenvalues is in
$\Ubar$.
\end{Lemma}

\begin{proof}
Assume that $A_0$ has at least two distinct eigenvalues. Note that
the eigenvalues of $A $ as an $n(k+1)\times n(k+1)$ matrix and the
eigenvalues of $A_0$ coincide. This can be seen by writing $C=t$
in an alternative basis in the block form
\begin{equation*}
\left(
 \begin{array}{ccccc}
   0 & I & 0 & 0 & 0\\
   0 & 0 & I & 0 & 0\\
   \vdots & \vdots & \vdots & \ddots & \vdots\\
   0 & 0 & 0 & 0 & I \\
   0 & 0 & 0 & 0 & 0 \\
 \end{array}
\right)
\end{equation*}
in which each entry is $n\times n$, for which correspondingly, $A
$ has the form
\begin{equation*}
\left(
 \begin{array}{ccccc}
   A_0 & A_1 & A_2 & \dots & A_k\\
   0 &A_ 0 & A_1 & \dots & A_{k-1}\\
   \vdots & \vdots & \ddots & \ddots & \vdots\\
   0 & 0 & \dots & A_0 & A_1 \\
   0 & 0 & \dots & 0 & A_0 \\
 \end{array}
\right)
\end{equation*} It follows from elementary considerations that since $A$, $B$, and $C$
commute and since $A$ has more than one eigenvalue, there exists
an $H \in \text{Gl}_{n(k+1)}(F)$ such that
\begin{equation*}
HA H^{-1} = \left(
               \begin{array}{cc}
                 P_1 & 0 \\
                 0 & P_2 \\
               \end{array}
             \right),
             \quad
HB H^{-1} = \left(
               \begin{array}{cc}
                 Q_1 & 0 \\
                 0 & Q_2 \\
               \end{array}
             \right),
             \quad
HCH^{-1} = \left(
               \begin{array}{cc}
                 C_1 & 0 \\
                 0 & C_2 \\
               \end{array}
             \right)
             \quad
\end{equation*}
where the upper left block is $r\times r$ and the lower right
block is $s \times s$, for some $r\ge 1$ and appropriate $s\ge 1$.
Moreover, we may further adjust $H$ so that $C_1$ and $C_2$ are in
Jordan canonical form. Since the Jordan form of $C$ is unique, it
follows that $r = u(k+1)$ and $s = v(k+1)$ for appropriate $u$ and
$v$, and $C_1$ consists of $u$ copies of $J_{k+1}$ while $C_2$
consists of $v$ copies of $J_{k+1}$. Hence, $(P_1, Q_1)$ is a
point on $\cjsu$ and $(P_2, Q_2)$ is a point on $\cjsv$.

We thus have a map $\cjsu \times \cjsv \rightarrow \cjs$
\begin{equation*}
(X_1, Y_1), (X_2, Y_2) \mapsto \left( H^{-1} \left(
               \begin{array}{cc}
                 X_1 & 0 \\
                 0 & X_2 \\
               \end{array}
             \right) H, H^{-1} \left(
               \begin{array}{cc}
                 Y_1 & 0 \\
                 0 & Y_2 \\
               \end{array}
             \right) H \right)
\end{equation*}
Write $Z$ for the image of this map, and note that $(A,B)$ is in
$Z$ and that $Z$ is irreducible by assumption on $\cjsu$ and
$\cjsv$. It is enough to show that $U\cap Z \neq \varnothing$ to
conclude that $Z$ is contained in $\Ubar$, since $U\cap Z$ would
then be dense in $Z$. But for this, take $X_1 = X_1(t)$ to be
$X_{1,0} + 0t + \cdots + 0t^{k}$ where $X_{1,0}$ is a $u\times u$
diagonal matrix with distinct diagonal entries $(\lambda_1, \dots,
\lambda_u)$, and $Y_1 = 0$. Similarly take $X_2 = X_2(t)$ to be
$X_{2,0} + 0t + \cdots + 0t^{k}$ where $X_{2,0}$ is a $v\times v$
diagonal matrix with distinct diagonal entries $(\lambda_{u+1},
\dots, \lambda_n)$, all distinct from those of $X_{1,0}$, and $Y_2
= 0$. Then, writing $H^{-1} \left(
               \begin{array}{cc}
                 X_1 & 0 \\
                 0 & X_2 \\
               \end{array}
             \right) H$ as $A_0 + A_1t + \cdots A_kt^k$, it is easy
to see that $A_0$ will have the distinct eigenvalues $(\lambda_1,
\dots, \lambda_n)$, so $A_0$ is $1$-regular.  Hence $U\cap Z$ is
nonempty.

\end{proof}

The next result is trivial but will be very useful.  First, write
$U'$ for the corresponding open subset of $\cjs$ where $B_0$ is
$1$-regular, then $U'$ is also irreducible by symmetric arguments.
Since $U\cap U'$ is nonempty, it follows immediately that $\Ubar =
\overline{U'} = \overline{U\cap U'}$.
\begin{Lemma} \label{lemma:prove closure with automorphisms}
Let $f$ be an automorphism of $\cjs$ such that $f(U) = U$ or
$f(U') = U'$ or  $f(U\cap U') = U\cap U'$. Then $(A,B)$ is in
$\Ubar$ iff $f(A ,B)$ is in $\Ubar$.
\end{Lemma}

Since $(A,B) \mapsto (A-\lambda I, B-\mu I)$ is such an
automorphism, we find:

\begin{Corollary} \label{cor:assume A0B0 nilpotent} Let $(A,B)\in \cjs$ be such that $A$ has the unique
eigenvalue $\lambda$ and $B$ has the unique eigenvalue $\mu$. Then
$(A,B)$ is in  $\Ubar$ if and only if $(A-\lambda I, B-\mu I)$ is
in
 $\Ubar$.
\end{Corollary}

As a result of Lemma \ref{lemma:assume A0B0 one eval} and
Corollary \ref{cor:assume A0B0 nilpotent} above, we may assume
that $A$ and $B$ are nilpotent, and since the eigenvalues of $A$
and $A_0$, and $B$ and $B_0$, coincide respectively, we may assume
that $A_0$ and $B_0$ are nilpotent while proving that
$(A(t),B(t))\in \Ubar$.  We will also need the following
reductions:

\begin{Corollary} \label{cor:can modify AB by polynomials} Let
$p(t)$ and $q(t)$ be polynomials in $F[t]$ of degree at most $k$,
and assume that $q(0) = 0$.  Then $(A(t),B(t)) \in \Ubar$ iff any
of the following occur:
\begin{enumerate} \item  $(B(t),A(t)) \in \Ubar$
\item $(A(t) + p(t)I, B(t)) \in \Ubar$
\item $(A(t), B(t) + p(t)I) \in \Ubar$
\item $(A(t), B(t) + p(t)A(t)) \in \Ubar$ \item $(A(t)(1 + q(t)), B(t))
\in \Ubar$ \end{enumerate}
\end{Corollary}

\begin{proof}
For the first assertion, note that $f\colon (A,B) \mapsto (B,A)$
is an automorphism of $\cjs$ that satisfies $f(U\cap U') = U\cap
U'$. Thus, Lemma \ref{lemma:prove closure with automorphisms}
applies.  The second, third and fourth assertions also follow from
Lemma \ref{lemma:prove closure with automorphisms} because the
obvious maps $f$ satisfy $f(U)=U$.  For the fifth assertion, for
fixed $q(t)$ with $q(0)=0$, note that $1+q(t)$ is invertible in
$F[t]/t^{k+1}$, with inverse $1-q'(t)$ for some $q'(t)$ with
$q'(0)=0$.  It follows that the map $f\colon (A(t),B(t)) \mapsto
(A(t)(1 +q(t)), B(t))$ is an automorphism of $\cjs$ in which the
coefficient of $t^0$ of $A(t)(1+q(t))$ is also $A_0$.  Hence $f(U)
= U$, so Lemma \ref{lemma:prove closure with automorphisms}
applies once again.
\end{proof}

Finally, we have the following reduction:
\begin{Corollary} \label{cor:assume A0B0 neq 0} Assume that
$(A(t),B(t))\in\Ubar$ whenever $A_0$ or $B_0$ is nonzero.  Then
$\Ubar = \cjs$.
\end{Corollary}

\begin{proof}
Take arbitrary $(A(t),B(t))  \in \cjs$ with $A_0 = B_0 = 0$.
First, for any $(P,Q) \in U$, $(\lambda P, \lambda Q)$ is also in
$U$ for nonzero $\lambda \in F$, hence, taking $\lambda = 0$, we
find $(0,0) \in \Ubar$.  Now assume that $A(t)$ is nonzero, and
assume that $A_0$, $\dots$, $A_{r-1}$ are zero for some $r\ge 1$
but $A_r$ is nonzero. Then writing $A'(t)$ for $A_r + t A_{r+1} +
\cdots t^{k-r} A_k$, we find $A'$ and $A$ commute, so for nonzero
$\lambda \in F$, $(A, B + \lambda A') \in \Ubar$ by the
hypothesis, so taking $\lambda=0$, we find $(A,B) \in \Ubar$.
\end{proof}

We now prove:
\begin{Theorem}
$\jsthree$ is irreducible.
\end{Theorem}
\begin{proof}
It is sufficient to show that $\cjsthree$ is irreducible.  For
this, it is sufficient, thanks to the various reductions above and
the irreducibility of $\jstwo$, to show that given $(A(t), B(t))
\in \cjsthree$ with $A_0$ nonzero and nilpotent, the pair $(A,B)$
is in $\Ubar$. By conjugating each of $A_i$ and $B_j$ with a fixed
$H \in \text{Gl}_n(F)$ we may assume (by Lemma \ref{lemma:prove
closure with automorphisms}) that $A_0$ is in Jordan form. If $A_0$
is $1$-regular we are done, so we may assume that $A_0$ is the
matrix $e_{1,2}$ with $1$ in the $(1,2)$ slot and zeros elsewhere.
We adopt the following notation for the entries: \begin{equation*}
A=\left(
\begin{array}{ccc}a(t)&b(t)&c(t)\\d(t)&e(t)&f(t)\\g(t)&h(t)&i(t)
\end{array}
\right), \quad B=\left(
\begin{array}{ccc}a'(t)&b'(t)&c'(t)\\d'(t)&e'(t)&f'(t)\\g'(t)&h'(t)&i'(t)
\end{array}
\right)
\end{equation*}

We may replace $A(t)$ by $A(t) - a(t)I$ and then $B(t)$ by $B(t) -
a'(t)I$ (Corollary \ref{cor:can modify AB by polynomials}), so we
can assume that $a(t)=a'(t)=0$. Since $b(0) = 1$, $b(t)$ is
invertible, with inverse of the form $1 + \text{higher terms}$, we
may replace $(A(t),B(t))$ by $(A (t)b(t)^{-1}, B(t))$ by Corollary
\ref{cor:can modify AB by polynomials}.  This does not change
$A_0$, but sets $b(t)=1$.  Next, we may replace $(A(t),B(t))$ by
$(A(t),B(t) - b'(t)A(t))$ and assume that $b'(t)=0$. The
commutativity relation $AB=BA$ implies that
$$d'(t)=c'(t)g(t)-c(t)g'(t),\quad e'(t)=c'(t)h(t)-c(t)h'(t),\quad f'(t)=c'(t)i(t)-c(t)i'(t)$$
and
$$g'(t)=i(t)h'(t)-i'(t)h(t)+c'(t)h(t)^2-c(t)h(t)h'(t)-e(t)h'(t).$$
We define the matrices
$$X=\left(
\begin{array}{ccc}0&0&0\\i(t)-e(t)-c(t)h(t)&1&0\\-h(t)&0&1
\end{array}
\right)\quad \mathrm{and}\quad Y=\left(
\begin{array}{ccc}0&0&0\\-h(t)c'(t)&0&c'(t)\\0&0&0
\end{array}
\right).$$ Then the above equations imply that $X$ and $Y$ commute
and $AY+XB=BX+YA$, so the matrices $A+ \lambda X$ and $B+\lambda
Y$ commute for each scalar $\lambda \in F$. For generic $\lambda$
the matrix $A+\lambda X$ has at least two distinct eigenvalues, so
for generic $\lambda \in F$ the pair $(A+\lambda X,B+\lambda Y)$
belongs to $\Ubar$ by Lemma \ref{lemma:assume A0B0 one eval} and
the irreducibility of $\jstwo$. Hence $(A,B)\in \Ubar$ and we are
done.

\end{proof}

We have the immediate corollary:
\begin{Corollary} \label{cor: algebra dimension 3(k+1)} Let   $N = 3(k+1)$
for $k=1,\dots$, and let $C\in M_{N}(F)$ be of the form $\lambda I
+ C'$, where $\lambda\in F$ and $C'$ is similar to the $N \times
N$ matrix consisting of three copies of $J_{k+1}$ along the
diagonal.  Then for any $A, B\in M_{N}(F)$ such that $A$, $B$, and
$C$ all commute, we have $\dim_F(F[A,B,C]) \le N$.
\end{Corollary}
\begin{proof}
The proof is standard, given the irreducibility of $\cjsthree$. We
may assume $\lambda = 0$ and $C'$ is in Jordan form since addition
of scalars and simultaneous conjugation does not change algebra
dimension.
On the open set $U$, $F[A,B,C] = F[A,C]$ by Theorem
\ref{theorem:characterization of A(t) with A0 1-regular}, Part
\ref{theorem:enum: B(t) is poly in A and t}. By classical results,
$\dim_F(F[A,C]) \le N$.  (In fact, by Theorem
\ref{theorem:characterization of A(t) with A0 1-regular}, Part
\ref{theorem_enum:F[A,t] has full dimension} $\dim_F(F[A,C]) = N$
on this open set.)
%
%
Since $\dim_F(F[A,B,C]) \le N$ is a closed set condition, it holds
on all of $\Ubar$. But the nonempty open set $U$ is dense in
$\cjsthree$ because of the irreducibility of $\cjsthree$. Hence
$\Ubar= \cjsthree$ and the theorem follows.
\end{proof}

\section{A Special Result when $\dim_F(F[A_0,B_0])=n$} We may combine the ideas in $\S2.2$ and in the proofs
of Theorem \ref{theorem:characterization of A(t) with A0
1-regular} and Corollary \ref{cor: algebra dimension 3(k+1)} to
establish the following:
\begin{Proposition} Let $A(t) = A_0 + A_1 t + \cdots + A_k t^k $ and $B(t) = B_0 + B_1
t + \cdots + B_k t^k$ commute in $M_n(F)[t]/t^{k+1}$, and assume
that $\dim_F(F[A_0,B_0])=n$.  Then viewing $M_n(F)[t]/t^{k+1}
\subset M_{n(k+1)}(F)$ as usual with $t$ corresponding to the
matrix $C$ consisting of $n$ copies of $J_{k+1}$ along the
diagonal, we have $\dim_F(F[A,B,C]) = n(k+1)$.

\end{Proposition}

\begin{proof} The same arguments as in $\S 2.2$ applied to the entire open subscheme
of $\mathcal{C}_{2,n}$ consisting of regular points show that the
open algebraic set $Y$ in $\cjs$ where $\dim_F(F[A_0,B_0])=n$ is
irreducible. Since $Y \supset U$ and $Y$ is irreducible, the
closure of $Y$ equals $\overline{U}$. We have seen in the proof of
Corollary \ref{cor: algebra dimension 3(k+1)} that
  $\dim_F(F[A,B,C]) \le
n(k+1)$ on $\overline{U}$, thus, this relation holds on  the
closure of $Y$. Now, as in the proof of Theorem
\ref{theorem:characterization of A(t) with A0 1-regular}
((\ref{theorem_enum:F[A,t] has full dimension}) $\Rightarrow$
(\ref{theorem_enum:1-regular})), we may define $E_i$ to be the
$F$-subspace of $F[A,B,t] \subset M_n(F)[t]/t^{k+1}$ of elements
of degree at least $i$.  We then have injective $F$-space maps
$E_i/E_{i+1} {\stackrel {\cdot t} \rightarrow} E_{i+1}/E_{i+2}$,
which coupled with the fact that $\dim_F(E_0/E_1) = \dim_F(F[A_0,B_0])
= n$, show that $\dim_F(F[A,B,C]) = \sum_{i=0}^k \dim_F \left(
E_i/E_{i+1}\right) \ge n(k+1)$.  Equality now follows.

\end{proof}

\end{document}